\documentclass[10pt,a4paper]{article}
\usepackage{titlesec}
\usepackage{enumitem}
\usepackage{url}
\usepackage{hyperref}
\titleformat*{\section}{\centering\bfseries\large}
\titleformat*{\subsection}{\centering\bfseries}
\usepackage{amsmath}
\usepackage[arrow, matrix, curve]{xy}
\usepackage{amssymb}
\usepackage{a4wide}
\usepackage{tikz}
\usepackage{amsthm}
\usepackage{latexsym}
\newtheorem{thm}{Theorem}[section]
\newtheorem{lemma}[thm]{Lemma}
\newtheorem{cor}[thm]{Corollary}

\theoremstyle{definition}
\newtheorem{definition}[thm]{Definition}
\theoremstyle{remark}
\newtheorem{remark}[thm]{Remark}

\DeclareMathOperator{\PGl}{PGL}
\DeclareMathOperator{\PSl}{PSL}
\DeclareMathOperator{\SL}{SL}
\DeclareMathOperator{\Gl}{GL}
\DeclareMathOperator{\Nrd}{Nrd}
\DeclareMathOperator{\aut}{Aut}
\newcommand{\CC}{\mathbb C}
\newcommand{\DD}{\mathbb D}
\newcommand{\PP}{\mathbb P}
\newcommand{\QQ}{\mathbb Q}
\newcommand{\RR}{\mathbb R}
\newcommand{\ZZ}{\mathbb Z}
\newcommand{\HH}{\mathbb H}

\newcommand{\NN}{\mathbb N}

\newcommand{\pP}{\mathfrak p}
\newcommand{\ppP}{\mathfrak P}
\newcommand{\qQ}{\mathfrak q}
\newcommand{\OO}{\mathcal O}

\begin{document}
\title{Varieties of general type with the same Betti numbers as $\PP^1\times \PP^1\times\ldots\times \PP^1$}
\author{Amir D\v{z}ambi\'c\footnote{Institut f\"ur Mathematik, Goethe-Universit\"at Frankfurt a.M., Postfach 111932, Fach 187, 60054 Frankfurt a.M. Email: dzambic@math.uni-frankfurt.de}}
%\date{}
\maketitle
\begin{abstract}
\noindent We study quotients $\Gamma\backslash \HH^n$ of the $n$-fold product of the upper half plane $\HH$ by irreducible and torsion-free lattices $\Gamma < \PSl_2(\RR)^n$ with the same Betti numbers as the $n$-fold product $(\PP^1)^n$ of projective lines. Such varieties are called \emph{fake products of projective lines} or \emph{fake $(\PP^1)^n$}. These are higher dimensional analogs of fake quadrics. In this paper we show that the number of fake $(\PP^1)^n$ is finite (independently of $n$), we give examples of fake $(\PP^1)^4$ and show that for $n>4$ there are no fake $(\PP^1)^n$ of the form $\Gamma\backslash \HH^n$ with $\Gamma$ contained in the norm-1 group of a maximal order of a quaternion algebra over a real number field.  
\end{abstract}

\section{Introduction}
After their classification of fake projective planes (see \cite{PrasadYeung07} and \cite{PrasadYeung10}) and the study of arithmetic fake projective spaces and fake Grassmannians (see \cite{PrasadYeung09}), in their paper \cite{PrasadYeung12}, G.~Prasad and S.-K.~Yeung introduced the general notion of a fake compact hermitian symmetric space and studied these spaces in detail. Let $X=G/K$ be a hermitian symmetric space of non-compact type and $\widehat{X}$ the compact dual of $X$. Then, by definition, a quotient $X_{\Gamma}=\Gamma\backslash X$ of $X$ by a cocompact and torsion-free discrete subgroup is a \emph{fake compact hermitian symmetric space or fake $\widehat{X}$} if $X_{\Gamma}$ has the same Betti numbers as $\widehat{X}$; such a fake $\widehat{X}$ is called irreducible, resp.~arithmetic, if $\Gamma$ is irreducible, resp.~arithmetic. One of the main results in \cite{PrasadYeung12} is the statement that there are no compact irreducible arithmetic fake hermitian symmetric spaces of type other than $A_{n}$ with $n\leq 4$. Here, the type of $X_{\Gamma}$ refers to the type of the irreducible factors $X_i$ of the universal covering $X=X_1\times\ldots \times X_s$ of $X_{\Gamma}$ according to Cartan's classification of irreducible hermitian symmetric spaces of non-compact type. Two-dimensional fake $A_2$ are exactly the fake projective planes, which are completely classified. In \cite{PrasadYeung09} we find examples of four-dimensional fake $A_2$, that is, fake products of projective planes $\PP^2\times \PP^2$. Note that the case of fake $A_1$ is not covered in \cite{PrasadYeung12}. By definition, a compact fake $A_1$ of dimension $n$ (called \emph{$n$-dimensional fake product of projective lines} in sequel) is a compact quotient $\Gamma\backslash \HH^n$ of the product of $n$ copies of the complex upper half plane $\HH$ by a cocompact torsion-free lattice $\Gamma\subset \PSl_2(\RR)^n$ with the same Betti numbers as the product $\PP^1\times\ldots\times\PP^1$ of $n$ copies of the complex projective line $\PP^1=\PP^1(\CC)$. The notion of a $n$-dimensional compact fake $A_1$ is meaningful only for $n\geq 2$ and in this case an irreducible fake $A_1$ is automatically arithmetic. Two-dimensional fake $A_1$ are also known as \emph{fake quadrics} (\cite{HirzebruchWerke1}, p.~779 f). There are many known irreducible as well as non-irreducible fake quadrics (see for instance \cite{Shavel78}, or \cite{dz11} for irreducible case and \cite{BCG08} for non-irreducible case). It is known that no fake products of projective lines of odd dimension are possible. In this note we study the existence questions of irreducible compact fake products of projective lines:\\

\noindent \textbf{Theorem A} (see Theorem \ref{finiteness})
There exists a constant $c>0$ such that for any integer $n>c$ there exists no irreducible fake $(\PP^1)^n$. The number of fake $(\PP^1)^n$ is finite.\\

\noindent Moreover we discuss the existence of irreducible fake products of projective lines with the fundamental group contained in the norm-1 group of a maximal order and show:\\

\noindent \textbf{Theorem B} (see Lemma \ref{leq6} and Theorems \ref{no6}, \ref{noexample3}, \ref{example1} and \ref{example2})
There exist two non-isomorphic 4-dimensional irreducible fake products of projective lines of the form $\Gamma\backslash \HH^4$ where $\Gamma$ is the norm-one group of a maximal order in totally indefinite quaternion algebras over the maximal totally real subfields of cyclotomic fields $\QQ(\xi_{20})$ and $\QQ(\xi_{24})$ of $20$-th and $24$-th roots of unity. Moreover, fake products of projective lines whose fundamental group is contained in a norm-one group of a maximal order exist only in dimension $2$ and $4$. In dimension $4$, the examples above are the only such examples up to isomorphism. \\

\noindent \textbf{Aknowledgments.} The author would like to thank Sai-Kee Yeung and Mikhail Belolipetsky for the helpful comments and discussions on the subject.
%presenting two examples, which are constructed as quotients of $\HH^4$ by an irreducible arithmetic lattice defined over maximal real subfields of cyclotomic fields $\QQ(\xi_{20})$ and $\QQ(\xi_{24})$ of $20$-th and $24$-th roots of unity, see Section \ref{examples}. Moreover we discuss the existence         

\section{Irreducible quotients of the polydisc}
Let $\HH$ denote the complex upper half plane and $\HH^n=\HH\times\ldots\times\HH$ the product of $n$ copies of $\HH$. Since $\HH$ is biholomorphically equivalent to the unit disc $\DD\subset \CC$, we will use the short term \emph{polydisc} to name $\HH^n$. The group $\Gl_2^+(\RR)^n=\Gl_2^+(\RR)\times\ldots\times\Gl_2^+(\RR)$, with $\Gl_2^+(\RR)$, the group of $2\times 2$-matrices with positive determinant, acts on $\HH^n$ componentwise as a group of linear fractional transformations. The quotient $\PGl_2^+(\RR)^n=\left(\Gl_2^+(\RR)/\RR^*\right)^n\cong\PSl_2(\RR)^n$ is identified with $\aut^+(\HH^n)$, the group of biholomorphic automorphisms, which preserve each factor. A lattice $\Gamma\subset \aut^+(\HH^n)$, that is, a discrete subgroup of finite covolume, is called \emph{irreducible}, if $\Gamma$ is not commensurable with a product $\Gamma_1\times\Gamma_2$ of two non-trivial lattices $\Gamma_1\subset \aut^+(\HH^r)$, $\Gamma_2\in\aut^+(\HH^{n-r})$ for some $0<r<n$. For $n\geq 2$, irreducible lattices can be constructed arithmetically in the following way:\\

Let $k$ be a totally real number field of degree $m=[k:\QQ]$ and let $A$ be a quaternion algebra over $k$. Choosing an ordering $v_{\infty,1},\ldots,v_{\infty,m}$ of the infinite places of $k$, assume that $A$ is unramified at first $n\leq m$ places $v_{\infty,1},\ldots,v_{\infty,n}$ and ramified at the remaining infinite places $v_{\infty,n+1},\ldots,v_{\infty,m}$. Additionally assume that $A$ ramifies at the finite places $\pP_1,\ldots,\pP_r$ of $k$. Let $d_A=\prod_{i=1}^r \pP_i$ denote the reduced discriminant of $A$. Then, $m-n+r\equiv 0\bmod 2$, and under this condition, $A$ is uniquely determined by the choice of $n$ unramified infinite places and the reduced discriminant $d_A$ up to isomorphism. We will write $A=A(k;m,n,d_A)=A(k;m,n,\pP_1,\ldots,\pP_r)$ to denote the isomorphism class of such $A$. If $\OO\subset A$ is an order in $A$, then, under the embeddings $A\hookrightarrow A\otimes_{k}k_{v_{\infty,i}}\cong M_2(\RR)$, for $i=1,\ldots, n$, the group $\OO^+=\{x\in \OO\mid \Nrd(x)\ \text{is}\ \text{totally positive unit in}\ k \}$ becomes a subgroup in $\Gl_2^+(\RR)^n$ where $\Nrd$ denotes the reduced norm. The group $\Gamma_{\OO}^+=\OO^+/k^{\ast}$ is a lattice in $\aut^+(\HH)^n$ and, moreover, $\Gamma_{\OO}^+$ is an irreducible lattice and so is every $\Gamma\subset \aut^+(\HH^n)$ which is commensurable with $\Gamma_{\OO}^+$. We say that $\Gamma\subset\aut^+(\HH^n)$ is an \emph{arithmetic lattice}, if there exits a number field $k$, a quaternion algebra $A=A(k;m,n,d_A)$ and a maximal order $\OO$ in $A$ such that $\Gamma$ is commensurable with $\Gamma_{\OO}^+$. For example $\Gamma_{\OO}^1=\{x\in \OO\mid \Nrd(x)=1 \}/\{\pm 1\}$ or any congruence subgroup in $\Gamma_{\OO}^+$ or $\Gamma_{\OO}^1$ is, by definition, arithmetic. By the celebrated theorem of Margulis, for $n\geq 2$, every irreducible lattice $\Gamma\in \aut^+(\HH)^n$ is an arithmetic lattice.\\

Let $\Gamma$ be an irreducible lattice in $\aut^+(\HH^n)$ commensurable with $\Gamma_{\OO}^+$, where $\OO$ is a maximal order in the quaternion algabra $A=A(k;m,n,d_A)$. As a corollary to Godement compactness criterion we have that $\Gamma$ is a cocompact lattice in $\aut^+(\HH^n)$ if and only if $A$ is a division algebra, or equivalently if and only if $\Gamma$ is not commensurable with a Hilbert modular group. Let $X_{\Gamma}=\Gamma\backslash \HH^n$ be the orbit space under the action of $\Gamma$ on $\HH^n$. Then, $X_{\Gamma}$ is a compact locally symmetric space. If $\Gamma$ is torsion-free, $X_{\Gamma}$ is a $n$-dimensional complex manifold and even more, in this case $X_{\Gamma}$ is a smooth projective variety whose canonical line bundle is ample. In particular, for torsion-free and cocompact lattices $\Gamma$, $X_{\Gamma}$ is a variety of general type.\\
By Hirzebruch's proportionality theorem, numerical invariants of $X_{\Gamma}$ are closely related to the numerical invariants of the compact dual $(\PP^1)^n$ of $\HH^n$ (see \cite{Hirzebruch58}). In fact, only the knowledge of the Euler number $e(X_{\Gamma})$ determines the complete Hodge diamond of $X_{\Gamma}$ (see \cite{MatsushimaShimura63} or Lemma \ref{matsshim} below). On the other hand, the Euler number can be computed by a volume formula which generalizes Siegel's formula for the volume of the fundamental domain of $\SL_2(\ZZ)$ in terms of Riemann zeta function. Namely, consider the $\aut^+(\HH)^n$-invariant volume form $\nu=(-2\pi)^{-n}\prod_{1\leq i\leq n} dx_i\wedge dy_i/y_i^2$ on $\HH^n$. Then, for torsion-free cocompact lattices $\Gamma\subset \aut^+(\HH^n)$ we have $e(X_{\Gamma})=vol(X_{\Gamma})$, where $vol(X_{\Gamma})$ denotes the volume of a $\Gamma$-fundamental domain in $\HH^n$ with respect to $\nu$.   

\begin{lemma}[see \cite{Vigneras76}]
\label{euler_zahl}
Let $k$ be a totally real number field of degree $m=[k:\QQ]$ and let $A=A(k;m,n,d_A)$ be a quaternion algebra
over $k$. Fix a maximal order $\OO\subset A$ and let $\Gamma\subset \PGl_2^+(\RR)^n$ be a subgroup commensurable with $\Gamma_{\OO}^1=\OO^1/\pm 1$. Then 
$$
vol(X_{\Gamma})=[\Gamma_{\OO}^1:\Gamma](-1)^{m+n}2^{n-m+1}\zeta_k(-1)\prod_{\pP\mid d_A}(N\pP-1)
$$
where $[\Gamma_{\OO}^1:\Gamma]=[\Gamma_{\OO}^1:\Gamma\cap\Gamma_{\OO}^1]/[\Gamma:\Gamma\cap \Gamma_{\OO}^1]$ denotes the generalized index between $\Gamma$ and $\Gamma_{\OO}^1$, $\zeta_k(\ )$ is the Dedekind zeta function and where for an integral prime ideal $\pP$ in $k$ $N\pP=|\OO_k/\pP|$ is the norm of $\pP$.
\end{lemma}

\subsection{Fake products of projective lines whose universal covering is the polydisc}

Let $\Gamma$ be a torsion-free cocompact irreducible lattice in $\aut^+(\HH^n)$. As mentioned above, the quotient $X_{\Gamma}=\Gamma\backslash \HH^n$ is a $n$-dimesional smooth projective variety of general type whose Betti numbers are closely related to the Betti numbers of $(\PP^1)^n$. 

\begin{definition}
We say that a compact quotient $X_{\Gamma}$ of the $n$-dimensional polydisc $\HH^n$ is a \emph{compact fake product of projective lines} (or simply \emph{fake $(\PP^1)^n$}) if $\Gamma$ is a cocompact and torsion-free lattice in $\aut^+(\HH^n)$ and $X_{\Gamma}=\Gamma\backslash \HH^n$ has the same Betti numbers as $(\PP^1)^n$. We say that $X_{\Gamma}$ is \emph{irreducible} if $\Gamma$ is an irreducible lattice in $\aut^+(\HH^n)$.  
\end{definition}

\begin{remark}
\begin{enumerate}
\item As already mentioned in the introduction, many examples of irreducible and non-irreducible fake $\PP^1\times \PP^1$ (fake quadrics) are known. The non-irreducible fake quadrics are examples of so-called Beauville surfaces. In higher dimensions, excluding the ``trivial'' cases of products of fake quadrics, no examples of fake $(\PP^1)^n$ seem to be known. It is an open question to construct such higher dimensional fake $( \PP^1)^n$ which are not products of fake $(\PP^1)^n$ in lower dimension. In this paper we will concentrate on irreducible fake $(\PP^1)^n$. But other ``non-trivial'' constructions are to be analyzed. In particular, this includes the construction of higher dimensional varieties isogenous to a product of curves (see \cite{Catanese00}) with given Betti numbers which are not products of Beauville surfaces. 

\item More generally, we could skip the condition on the universal covering, and define a fake $(\PP^1)^n$ as a variety of general type with the same Betti numbers as $(\PP^1)^n$. Here an open question is to identify the universal covering of such a variety; is the universal covering of a variety of general type with the same Betti numbers as $(\PP^1)^n$ always $\HH^n$? (compare \cite{HirzebruchWerke1}, p.~780).     
\end{enumerate}
\end{remark}

A theorem of Matsushima-Shimura gives a characterization of irreducible fake products of projective lines:

\begin{lemma}[\cite{MatsushimaShimura63}, Theorem 7.2 and Theorem 7.3]
\label{matsshim}
Let $\Gamma$ be an irreducible and torsion-free lattice in $\aut^+(\HH^n)$ and let $b_i(X_{\Gamma})$ denote the $i$-th Betti number of $X_{\Gamma}$. Then
\begin{itemize}
\item For $i\neq n$, $b_i(X_{\Gamma})=b_i(\left (\PP^1 \right)^n)$
\item $b_n(X_{\Gamma})=b_n(\left (\PP^1 \right)^n)+2^n h^{n,0}(X_{\Gamma})$, where $h^{n,0}(X_{\Gamma})=\dim H^0(X_{\Gamma},\Omega_{X_{\Gamma}}^n)$.
\end{itemize}
If $n=\dim X_{\Gamma}$ is odd, $X_{\Gamma}$ cannot be a fake product of projective lines. For $n$ even, $X_{\Gamma}$ is a fake product of projective lines if and only if the arithmetic genus $\chi(X_{\Gamma})$ equals to $1$.
\end{lemma}
Since only $b_n(X_{\Gamma})$ may be different from $b_n((\PP^1)^n)$, we can characterize fake products of projective lines also by the value of the Euler number $e(X_{\Gamma})$.
\begin{cor}
For even $n$ a quotient $X_{\Gamma}$ is a fake $(\PP^1)^n$ if and only if $e(X_{\Gamma})=e((\PP^1)^n)=2^n$. 
\end{cor}

\section{Finiteness results on irreducible fake $(\PP^1)^n$}
In this section we will prove the finiteness result which states that the dimension of a fake $(\PP^1)^n$ is bounded by an absolute constant from above and that in each dimension there are only finitely many fake $(\PP^1)^n$. Moreover we will prove the non-existence of irreducible fake $(\PP^1)^n$ $X_{\Gamma}$ whose fundamental group $\Gamma$ is contained in a norm-1 group $\Gamma_{\OO}^1$ of maximal order of a quaternion algebra over a real number field in all dimensions $n>4$. In fact, we will list all the isomorphism classes of quaternion algebras which contain a torsion-free group $\Gamma\leq \Gamma_{\OO}^1$ such that $X_{\Gamma}$ is a $4$-dimensional fake product of projective lines. 

Recall that the fundamental group $\Gamma$ of an irreducible $n$-dimensional fake product of projective lines is an arithmetic lattice commensurable with $\Gamma_{\OO}^1$, where $\OO$ is a maximal order in a quaternion algebra $A=A(k;m,n,d_A)$ over a totally real number field $k$ of degree $m$ with prescribed ramification at infinite places. If $\Gamma$ is contained in a lattice $\Delta$ we have $vol(X_{\Delta})\leq vol(X_{\Gamma})=e(X_{\Gamma})=2^n$. Every $\Gamma$ is contained in a maximal lattice $\Delta$ with finite index, and in the first step we will show that for $n>c$ ($c$ a constant) every maximal irreducible lattice $\Delta<\PSl_2(\RR)^n$ satisfies $vol(\Delta)>2^n$. In the second step we shall show that in a fixed dimension $n$ there exist only finitely many conjugacy classes of maximal (irreducible) lattices $\Delta$ in $\PSl_2(\RR)^n$ such that $vol(\Delta)\leq 2^n$. For this purpose we will intensively make use of results in \cite{Borel81} and \cite{ChinburgFriedman86} where volumes and commensurabilities between arithmetic lattices of $\PGl_2(\RR)^a\times \PGl_2(\CC)^b$ are studied in great detail. The most relevant results on for our purposes are summarized in the following 
\begin{lemma}
\label{borelmax}
Let $k$ be totally real number field of degree $m$ and $A=A(k;m,n,\pP_1,\ldots,\pP_r)$ a quaternion algebra over $k$ with $m\geq n\geq 1$ (that is, $A$ satisfies the Eichler condition). 
\begin{enumerate}
\item (compare \cite[4.9, 8.4-8.6]{Borel81}, \cite[Lemma 2.1]{ChinburgFriedman86}) The maximal arithmetic subgroup of $A^+/k^{\ast}=\{x\in A\mid \Nrd(x)\ \text{is totally positive}\}/k^{\ast}$ which contains $\Gamma_{\OO}^1$ for a maximal order $\OO$ in $A$ is $N\Gamma_{\OO}^+=\{x\in A^+\mid x\OO x^{-1}=\OO\}/k^{\ast}$. The index of $\Gamma_{\OO}^1$ in $N\Gamma_{\OO}^+$ is $[N\Gamma_{\OO}^+:\Gamma_{\OO}^1]=2^r[k_A:k]$ where $k_A$ is the maximal abelian extension of $k$ which is unramified at all finite places of $k$ and such that its Galois group $Gal(k_A/k)$ is elementary 2-abelian and in which all the prime ideals $\pP_1,\ldots,\pP_r$ are totally split.
\item (compare \cite[4.4, 5.3]{Borel81}, \cite[11.4]{MaclachlanReid03}) Let $S$ be a finite set of primes of $k$ such that $\pP_i\notin S$ for $i=1,\ldots, r$ and let $\OO(S)$ be the Eichler order of level $\prod_{\qQ\in S} \qQ$. Let $\Gamma_{S,\OO}^+$ be the normalizer of $\OO(S)$ in $A^+/k^{\ast}$. Then every maximal arithmetic subgroup of $A^+/k^{\ast}$ is (a conjugate of) a lattice of the form $\Gamma_{S,\OO}^+$ with $\Gamma_{\emptyset,\OO}^+=N\Gamma_{\OO'}^+$ for some maximal order $\OO'$. There exists an integer $0\leq s\leq |S|$ such that the generalized index  $[N\Gamma_{\OO}^+:\Gamma_{S,\OO}^+]$ equals $2^{-s}\prod_{\qQ\in S}(N\qQ+1)$. For any $S$ we have $vol(X_{N\Gamma_{\OO}^+})\leq vol(\Gamma_{S,\OO}^+)$ with equality if and only if $S=\emptyset$. There exist only finitely many conjugacy classes of maximal lattices $\Gamma_{S,\OO}^+$ commensurable with a given $N\Gamma_{\OO}^+$ such that $[N\Gamma_{\OO}^+:\Gamma_{S,\OO}^+]\leq c$ for any given $c>0$.  

\item (compare \cite[5.4]{Borel81}) Let $\mathcal C(k,A)$ be the set of all irreducible lattices in $\PSl_2(\RR)^n$ which are commensurable with $\Gamma_{\mathcal O}^1$ for some maximal order $\mathcal O$ in $A/k$. Then the function $\Gamma\mapsto vol(X_{\Gamma})$ takes its minimum on a conjugate of $N\Gamma_{\OO}^+$. Let $e$ be the number of dyadic places of $k$ (finite places lying over $2$) not dividing the reduced discriminant of $A$. Then $vol(\Gamma)$ is a positive integral multiple of $2^{-e}vol(X_{N\Gamma_{\OO}^+})$ for any $\Gamma$ in $\mathcal C(k,A)$.
\end{enumerate}
\end{lemma} 

\noindent The finiteness of the set of irreducible fake $(\PP^1)^n$ follows immediately from the following Theorem.  

\begin{thm}
\label{finiteness}
There exists a constant $c>0$ such that if $N\Gamma_{\OO}^+\hookrightarrow \PSl_2(\RR)^n$ is a maximal lattice with $vol(N\Gamma_{\OO}^+)\leq 2^n$ then $n\leq c$. For each $n$ there are only finitely many conjugacy classes in $\PSl_2(\RR)^n$ of maximal lattices $\Delta$ such that $vol(X_{\Delta})\leq 2^n$. 
\end{thm}

\begin{proof}
By Lemma \ref{borelmax} and Lemma \ref{euler_zahl}, the volume of $X_{N\Gamma_{\OO}^+}$ is given by  

$$
vol(X_{N\Gamma_{\OO}^+})=(-1)^{n+m}\frac{2^{n-m+1}\zeta_k(-1)}{2^r[k_A:k]}\prod_{i=1}^r(N\pP_i-1)
$$ 
Let $k'_A$ be the abelian extension of $k$ with the same properties as $k_A$ but which is additionally unramified also at all infinite places of $k$. Then (\cite[Proposition 2.1]{ChinburgFriedman86}) 
$$
[k_A:k]^{-1}=\frac{[\mathfrak o_k^{\ast}:\mathfrak o_{k,+}^{\ast}]}{2^m}[k'_A:k]^{-1}.
$$
The advantage of considering $k'_A$ instead of $k_A$ is the fact that $[k'_A:k]$ divides the class number of $k$. 
Now recall the functional equation of the Dedekind zeta function by which for a totally real number field $k$ of degree $m$ we have 

\begin{equation}
\label{functionalequation}
\zeta_k(-1)=(-1)^m 2^{-m}\pi^{-2m}d_k^{3/2}\zeta_k(2),
\end{equation}
where $d_k$ is the discriminant of $k$. Keeping in mind that $n$-dimensional fake products of projective lines exist only for even $n$ we obtain

$$
vol(X_{N\Gamma_{\OO}^+})=\frac{d_k^{3/2}\zeta_k(2)}{2^{2m-n-1+t'}[k_A:k]\pi^{2m}}\prod_{\substack{i\in\{1,\ldots,r\}\\  N\pP_i\neq 2}} \frac{N\pP_i-1}{2}
$$

\noindent where $t'$ is the number of those primes $\pP_i$ which are ramified in $A$ and such that $N\pP_i=2$. Let $t$ be the number of primes of $k$ which divide $2$. Then, $t'\leq t$. Define 
\begin{equation}
\label{gkb}
g(k,A)=\frac{d_k^{3/2}\zeta_k(2)}{2^{2m-1+t}\pi^{2m}[k_A:k]}
=\frac{d_k^{3/2}\zeta_k(2)[\mathfrak o_k^{\ast}:\mathfrak o^{\ast}_{k,+}]}{2^{3m-1+t}\pi^{2m}[k'_A:k]}.
\end{equation}
Then, 
\begin{equation}
\label{gkbvol}
vol(X_{N\Gamma_{\OO}^+})=2^{t-t'+n}g(k,A)\prod_{\substack{i\in\{1,\ldots,r\}\\  N\pP_i\neq 2}} \frac{N\pP_i-1}{2}
\end{equation}

Assume now that $2^n\geq vol(X_{N\Gamma_{\OO}^+})$. As $\prod_{\substack{i\in\{1,\ldots,r\}\\  N\pP_i\neq 2}}\frac{N\pP_i-1}{2}\geq 1$ and $t\geq t'$, the inequality 
\begin{equation}
\label{22n}
1\geq g(k,A)
\end{equation}
follows. The main ingredient at this point is a lower bound for $g(k,A)$ depending on $m$ and $[k_A:k]$ obtained by Chinburg and Friedman. In \cite[Lemma 3.2]{ChinburgFriedman86} we find the inequality
\begin{equation}
\label{chfrunten}
g(k,A)> 0.142\exp\left( 0.051\cdot m-\frac{19.0745}{[k'_A:k]}\right).
\end{equation}
The right hand side of (\ref{chfrunten}) tends to infinity for $m\rightarrow \infty$ . Thus, the condition (\ref{22n}) implies that $m$ is bounded from above and since the dimension, or equivalently the number of unramified infinite places $n$ in $A$ satisfies $n\leq m$, also $n$ is bounded. Let now $m$ be fixed. We want to show that in this case the discriminant of $k$ and reduced discriminant of $A$ are bounded. In order to do so, we first replace $[k'_A:k]$ in (\ref{gkb}) by the class number $h_k$ of $k$. Also we use the estimate on $h_k$ provided by the Brauer-Siegel theorem. Let $s>1$ be a real number, $k$ a totally real number field of degree $m$ with discriminant $d_k$, regulator $R_k$ and the class number $h_k$. Then, by the Theorem of Brauer-Siegel 
\begin{equation}
\label{BrauerSiegel}
h_kR_k\leq 2^{1-m}s(s-1)\Gamma(s/2)^m\left(\frac{d_k}{\pi^m}\right)^{s/2}\zeta_k(s).
\end{equation}
Using a lower bound for the regulator of the form $R_k\geq c_1\exp(c_2 m)$ (see for instance \cite[p.375]{Zimmert81} for an explicit choice of the constants $c_1$ and $c_2$) and choosing $s=2$ in (\ref{BrauerSiegel}) we obtain an upper bound 
$$
h_k\leq \frac{2^{2-m}d_k\zeta_k(2)}{\pi^mc_1\exp(c_2 m)}
$$
Plugging this into (\ref{22n}) we obtain
\begin{align*}
1\geq g(k,A)\geq &\frac{d_k^{3/2}\zeta_k(2)[\mathfrak o_k^{\ast}:\mathfrak o^{\ast}_{k,+}]}{2^{3m-1+t}\pi^{2m}h_k}\\
\geq & \frac{d_k^{1/2}[\mathfrak o_k^{\ast}:\mathfrak o^{\ast}_{k,+}]c_1\exp(c_2m)}{\pi^{m}2^{2m+t+1}}
\end{align*}
As $[\mathfrak o_k^{\ast}:\mathfrak o^{\ast}_{k,+}]\geq 2$ and $t\leq m$ the above inequalities give
$$
d_k^{1/2}\leq \frac{\pi^m2^{3m}}{c_1\exp(c_2m)}.
$$
For fixed $m$, the right hand side of the last inequality is constant and it follows that the discriminant $d_k$ is bounded. It follows that there are only finitely many possible totally real fields which may serve as fields defining the commensurability class $\mathcal C(k,A)$ of $N\Gamma_{\OO}^+$. If $k$ is fixed, from the equation (\ref{gkbvol}) it easily follows that if $vol(N\Gamma_{\OO}^+)$ is bounded, the reduced discriminant $d_A$ is also bounded. As the reduced discriminant determines the isomorphism class of $A$ (notice that the ramification behavior at infinite places is fixed), there are only finitely many isomorphism classes of quaternion algebras which define lattices of fake $(\PP^1)^n$. Finally, there are only finitely many non-conjugate maximal orders inside a given quaternion algebra $A$, hence only finitely many non-conjugate maximal lattices of type $N\Gamma_{\OO}^+$ inside the commensurability class $\mathcal C(k,A)$. It follows from Lemma \ref{borelmax}(2.) there are only finitely many maximal lattices $\Gamma_{S,\OO}^+\in \mathcal C(k,A)$ with $vol(X_{\Gamma_{S,\OO}^+})\leq 2^n$.   
\end{proof}

A natural question which arises from the Theorem \ref{finiteness} is that on effectivity: Is it possible to list all fake $(\PP^1)^n$? Less ambitious we could ask, to what extent one can make precise the bounds on invariants $n,m,d_k, d_A$ which belong to fake $(\PP^1)^n$? Certainly, a careful analysis of the proof of Theorem \ref{finiteness} will provide bounds on the above invariants. For instance, the equation (\ref{chfrunten}) implies that the dimension $n$ of a fake $(\PP^1)^n$ is less or equal $412$. In fact the bounds which we get are not expected to be very precise. The main hurdle is the invariant $[k_A:k]$ or rather $[k'_A:k]$ associated with $A$ for which we apparently miss good bounds and which we are forced to estimate by the class number. Some observations on this invariant have been also made by Belolipetsky and Linowitz \cite{BelolipetskyLinowitz} in connection with the enumeration of fields of definition of arithmetic Kleinian reflection groups.             
 
In order to get a first impression on how big the number of fake $(\PP^1)^n$ can be, we will now concentrate our attention to fake $(\PP^1)^n$ whose fundamental group is contained in the norm one group of a maximal order. This is a class of arithmetic groups which is much more accessible than the general ones, since the critical invariant $[k’_A:k]$ does not appear in the volume formula of such lattices. We will see soon that fake $(\PP^1)^n$ with such a fundamental group are very rare. So, from now on let us assume that:

\begin{equation}
\label{assumption} 
\textit{The lattice $\Gamma$ is contained in the norm-1 group $\Gamma_{\OO}^{1}$.}
\end{equation}

\begin{remark}
Note that the above assumption is indeed restrictive. Typical examples of lattices which are not contained in $\Gamma_{\OO}^1$ are the normalizer of maximal and Eichler orders.  
\end{remark} 
Under the assumption (\ref{assumption}) the inequality $2^ n=e(X_{\Gamma})\geq e(X_{\Gamma_{\OO}^1})$ holds.

The functional equation (\ref{functionalequation}) for the Dedekind zeta function and the Lemma \ref{euler_zahl} imply directly the equality

\begin{equation}
\label{2hochn}
e(X_{\Gamma})=2^n\geq e(X_{\Gamma_{\OO}^1})=2^{n-2m+1}\pi^{-2m}d_k^{3/2}\zeta_k(2)\prod_{\pP\mid d_A}(N\pP-1). 
\end{equation}
As next we observe that $\prod_{\pP\mid d_A}(N\pP-1)\geq 1$ as well as $\zeta_k(2)>\zeta_{\QQ}(2m)>1$. The last inequality $\zeta_k(2)>\zeta_{\QQ}(2m)$ (or more generally $\zeta_k(s)>\zeta_{\QQ}(ms)$ for any $s>1$) follows easily from the Euler product representation. Namely, 
$$
\zeta_k(2)=\prod_{\pP} \left( 1-\frac{1}{N\pP^2}\right)^{-1}=\prod_{p} \prod_{\pP\mid p} \left( 1-\frac{1}{N\pP^2}\right)^{-1},
$$ 
where $\pP$ runs through the set of non-zero prime ideals in $\OO_k$ and $p$ through all rational primes. Suppose that $\pP_1,\ldots,\pP_t$ are the prime divisors of $p\OO_k$ and $N\pP_i=p^{f_i}$. Then
$$
\prod_{\pP\mid p} \left( 1-\frac{1}{N\pP^2}\right)^{-1}=\prod_{i=1}^t\left( \frac{p^{2f_i}}{p^{2f_i}-1}\right)=\frac{p^{2(f_1+\ldots+f_t)}}{\prod_{i=1}^t (p^{2f_i}-1)}.
$$
We have an obvious inequality $\prod_{i=1}^t (p^{2f_i}-1)<p^{2(f_1+\ldots+f_t)}-1$. Hence

$$
\prod_{\pP\mid p} \left( 1-\frac{1}{N\pP^2}\right)^{-1}>\frac{p^{2(f_1+\ldots+f_t)}}{p^{2(f_1+\ldots+f_t)}-1}=\left(1-\frac{1}{p^{2(f_1+\ldots+f_t)}} \right)^{-1}.
$$
Finally, the fundamental identity for prime ideals, \cite[Proposition I.(8.2)]{NeukirchANT}, implies that $f_1+\ldots+f_t\leq m$ and we have $\prod_{\pP\mid p} \left( 1-\frac{1}{N\pP^2}\right)^{-1}>\left(1-\frac{1}{p^{2m}} \right)^{-1}$ and the stated inequality follows (Obviously, $\zeta_{\QQ}(s)=1+1/2^s+\ldots\ldots>1$ for any real $s>1$). 

Using the above inequalities, the equation (\ref{2hochn}) implies the relation 
$$
2^n>2^{n-2m+1}\pi^{-2m}d_k^{3/2}
$$
or equivalently 
\begin{equation}
\label{dkoben}
d_k<\left(\frac{(2\pi)^{2m}}{2}\right)^{2/3}.
\end{equation}
For any totally real field $K$ of degree $m$ and discriminant $d_K$ let $\delta_{K}=d_K^{1/m}$ be its so-called root discriminant and let $$\delta^{r}_{min}(m)=\min\{\delta_K\mid K\ \text{is totally real of degree}\ m\}.$$ 
The equation (\ref{dkoben}) implies that for our purposes relevant fields $k$ must satisfy $\delta_k<f(m)=\frac{(2\pi)^{4/3}}{2^{2/3m}}$. The function $f$ is increasing in $m$ but $f(m)<(2\pi)^{4/3}<11.6$ for all $m$.  A.~Odlyzko proved lower bounds for $\delta^{r}_{min}(m)$ and from his work we know that for any $m\geq 9$ the inequality $\delta^r_{min}(m)>11.823$ holds (see \cite{Odlyzkohomepage}). We conclude that $m\leq 8$. But we can improve this knowing the exact value of the minimal root discriminant of a totally real field of degree $\leq 8$ which has been determined by J. Voight in \cite[Table 3]{Voight08}. The values of $\delta^r_{min}(m)$ and $f(m)$ are compared in the Table \ref{fm}.
\begin{table}[h]
\begin{center}
\begin{tabular}{|c|c|c|c|c|c|}
\hline
$m$ & $8$ & $7$ & $6$ & $5$ & $4$  \\
\hline
$\delta^r_{min}(m)$ &11.385 &11.051 & 8.182 & 6.809 & 5.189   \\
\hline
$f(m)$ & 10.943& 10.853&10.734 & 10.570&  10.329 \\
\hline
\end{tabular} 
\end{center}
\caption{Root discriminant bounds}
\label{fm}
\end{table}

Altogether, we have
\begin{lemma}
\label{leq6}
If $\Gamma_{\OO}^1$, with $\OO$ a maximal order in the quaternion algebra $A(k;m,n,d_{A})$ as above, contains a lattice of a fake $(\PP^1)^n$, then $n\leq m=[k:\QQ]\leq 6$.   
\end{lemma}
Since the dimension $n$ of an arithmetic fake product is always less or equal the degree $m$ of the center field $k$ of the defining quaternion algebra, any irreducible fake $(\PP^1)^n$ whose fundamental group satisfies condition (\ref{assumption}) is either $2$-,$4$-, or $6$-dimensional. As already remarked, the $2$-dimensional examples are discussed in \cite{dz11} (see also references therein) in greater generality. 
Moreover we can prove
\begin{thm}
\label{no6}
There are no irreducible fake $(\PP^1)^6$ $X_{\Gamma}$ such that $\Gamma$ satisfies the condition (\ref{assumption}).
\end{thm}
\begin{proof}
Assume that $X_{\Gamma}$ is an irreducible fake $(\PP^1)^6$ such that $\Gamma<\Gamma_{\OO}^1$ and $\OO^1$ is a maximal order in $A=A(k;6,6,d_A)$. According to Lemma \ref{leq6} the totally real number field $k$ is of degree $6$ with discriminant $d_k\leq (10.734)^6\approx1529570.6$. Finally, $vol(X_{\Gamma_{\OO}^1})=2\zeta_k(-1)\cdot \prod_{\pP\mid d_A}(N\pP-1)\leq 2^6$. As $\Gamma$ is cocompact, $A$ is a division algebra and therefore there must be at least one finite place of $k$ ramified in $A$ (by assumption, $A$ is unramified at all infinite places of $k$). Since the number of ramified places is even (by the reciprocity law for Hilbert-symbols \cite[Corollaire 3.3, p.~76]{VignerasAlgebres}), there must be at least two such finite places. In Table \ref{grad6} (see Section \ref{tables}), we collected the needed invariants of all totally real sextic fields with discriminant $<1529570.6$ which were computed with PARI. The list of these fields has been produced by J.~Voight (see \cite{Voighthomepage}). With the knowledge of all these values, case by case analysis shows that there is no totally indefinite division quaternion algebra $A$ over a sextic real field satisfying $2\zeta_k(-1)\cdot \prod_{\pP\mid d_A}(N\pP-1)\leq 2^6$. We note that the last condition bounds the value of $\prod_{\pP\mid d_A}(N\pP-1)=\prod_{\pP\mid d_A} (p^{f(\pP/p)}-1)$, with $f(\pP/p)$ the inertia degree of $\pP$, and thus restricts the set of possible prime ideals $\pP$ at which $A$ ramifies.          
\end{proof}

We will now focus on $4$-dimensional fake products. 

\begin{thm}
\label{candidates}
Let $ A=A(k;m,4,d_A)$ be a quaternion algebra over a totally real number field $k$ unramified at $4$ archimedean places of $k$ and $\OO$ a maximal order in $A$. Assume that $\Gamma<\Gamma_{\OO}^1$ is a finite index torsion free subgroup such that $X_{\Gamma}=\Gamma\backslash \HH^4$ is a fake product of projective lines. Then the pair $(k, d_A)$ (determining the isomorphism class of $A$ uniquely) belongs to the following list

\begin{table}[h!]
\label{deg4kandidaten}
\centering
\begin{tabular}{|c|c|c|c|}
\hline
$d_k$ & defining polynomial & $\zeta_k(-1)$ & $d_A$ \\
\hline
$1957$ & $x^4 - 4x^2 - x + 1$ &  $2/3$ & $\pP_3\pP_7$\\
\hline
$2000$ & $x^4 - 5x^2 + 5$ &  $2/3$ & $\pP_2\pP_5$ \\
\hline
$2304$ & $x^4 - 4x^2 + 1$ &  $1$ & $\pP_2\pP_3$\\
\hline
$38569 $ & $ x^5 - 5x^3 + 4x - 1$ & $-8/3$ & $\pP_7$\\
\hline
$106069$ & $ x^5-2x^4 -4x^3 + 7x^2 + 3x-4$ & $-16$ & $\pP_2$\\
\hline
$ 453789$ & $x^6-x^5-6x^4+6x^3+8x^2-8x+1$ & $ 16/3$ & $\emptyset$ \\
\hline
$ 1387029 $ & $x^6 - 3x^5 - 2x^4 + 9x^3 - x^2 - 4x + 1  $ & $ 32 $ & $\emptyset$ \\
\hline
$ 1397493 $ & $x^6 - 3x^5 - 3x^4 + 10x^3 + 3x^2 - 6x + 1  $ & $ 32 $ & $\emptyset$ \\
\hline
\end{tabular} 
\label{liste}  

\end{table}

There, the reduced discriminant $d_A=\prod \pP_p$ is a (possibly empty) formal product of \underline{suitably chosen} prime ideals $\pP_p$ lying over a rational prime $p$.  
\end{thm}

\begin{remark}
To specify what ``suitably chosen`` in the above Theorem \ref{candidates} means, let us consider an example. Let $k$ be the totally real field of degree $4$ and discriminant $d_k=38569$. Then there are two prime ideals of $k$ lying over $7$, but one with inertia degree $1$ and the other with inertia degree $4$ (see Table \ref{grad5}, Section \ref{tables}). Only the first prime ideal can be taken as the prime where $A$, the quaternion algebra defining a fake $(\PP^1)^4$, ramifies, but not the latter.
\end{remark}

\begin{proof}[Proof of Theorem \ref{candidates}]
Let $A$ and $\Gamma<\Gamma_{\OO}^1$ as above be given. Then by Lemma \ref{leq6}, $m=[k:\QQ]\leq 6$ and
the root discriminant $\delta_k$ satisfies $\delta_k\leq f(m)$ 
with the value $f(m)$ from the Table $\ref{fm}$.
We know that $e(X_{\Gamma})=16=[\Gamma_{\OO}^1:\Gamma]e(X_{\Gamma_{\OO}^1})=2^{5-m+1}\zeta_k(-1)\cdot \text{integer}$. Hence, $k$ satisfies the condition 
\begin{equation}
\label{integral}
\frac{16}{2^{5-m}\zeta_k(-1)}\in \NN.
\end{equation}
In the already mentioned Table \ref{grad6} from Section \ref{tables} we find all the sextic fields $k$ satisfying $\delta_k\leq f(6)$. Additionally, Tables \ref{grad5} and \ref{grad4} (see Section \ref{tables}) contain all the totally real quintic and quartic fields with $\delta_k\leq f(m)$ and satisfying the condition (\ref{integral}). As next, note that for each finite place $\pP$ of $k$ which divides the reduced discriminant $d_A$, the value $N\pP-1=p^{f(\pP/p)}-1$ divides $\frac{16}{2^{5-m}\zeta_k(-1)}$. The relevant values for $f(\pP/p)$ are given in Tables \ref{grad6}, \ref{grad5} and \ref{grad4}. Finally, for $m=4,5$, the quaternion algebra (which is assumed to be a skew field) has to be ramified at least at one finite place of $k$ and if $m=4$, $A$ ramifies at least at two finite places. 

\end{proof}

\section{Examples}
\label{examples}
Let $\xi=\xi_{20}$ be a primitive $20$-th root of unity, $K=\QQ(\xi)$ the corresponding cyclotomic field and $k=\QQ(\xi+\xi^{-1})=\QQ(\cos(\pi/10))$ the maximal totally real subfield of $K$. The field $k$ is defined by the polynomial $P(x)=x^4-5x^2+5$. We summarize some relevant facts about $k$.
\begin{lemma}
\label{zerlegung1}
The field $k$ is a totally real abelian quartic field of discriminant $d_k=2000$. The small rational primes have the following prime ideal decomposition in $k$:
\begin{itemize}
\item $2\OO_k=\pP_2^2$, where $\pP_2$ is a prime ideal in $\OO_k$ with norm $N\pP_2=4$
\item $3\OO_k=\pP_3$ is inert.
\item $5\OO_k=\pP_5^4$ is totally ramified in $k$.
\end{itemize}
The value of Dedekind zeta function $\zeta_k(s)$ at $s=-1$ is $\zeta_k(-1)=\frac{2}{3}$.
\end{lemma}
\begin{proof}
Either one consults the Table \ref{grad4} in Section \ref{tables} or one applies directly some of the well-known facts about cyclotomic fields. Let us briefly explain the second approach:\\ Let $\xi_n$ be a primitive $n$-th root of unity. Then, the extension $\QQ(\xi_n)/\QQ(\xi_n+\xi_n^{-1})$ is unramified at all the finite places of $k$ for $n$ not a prime power (\cite[Proposition 2.15]{WashingtonCF}). Also, a rational prime $p$ factorizes in $\QQ(\xi_n)$ as 
$$
p\OO_{\QQ(\xi_n)}=(\mathfrak Q_1\ldots\mathfrak Q_r)^{\varphi(p^{\nu_p(n)})}
$$
where all $\mathfrak Q_i$ are of residual degree $f_p=\min\{f\mid p^f\equiv 1\bmod n/p^{\nu_p(n)}\}$ and $\nu_p(\ )$ is the normalized $p$-valuation (\cite[Proposition I.(10.3)]{NeukirchANT}). Recall again the fundamental identity for prime ideals \cite[Proposition I.(8.2)]{NeukirchANT} which in this case states that $f_p\varphi(p^{\nu_p(n)})r=\varphi(n)$. It follows for $\xi=\xi_{20}$ that the principal ideal $2\OO_{\QQ(\xi)}=\ppP_2^2$ is a square of a prime ideal in $\OO_{\QQ(\xi)}$. As $\QQ(\xi)/k$ is unramified, the stated decomposition of $2$ in $\OO_k$ follows. The above shows also that $3$ is inert in $\OO_{\QQ(\xi)}$ hence inert in $\OO_k$. Finally, we find that $5\OO_{\QQ(\xi)}=(\ppP\ppP')^4$ is a product of two prime powers. Again, as the extension $\QQ(\xi)/k$ is unramified, the fundamental identity allows only the stated possibility for the decomposition of $5$ in $\OO_k$. For the computation of the zeta value, one can use the formula for $\zeta_k(-1)$ as the product of values $L(-1,\chi)$, where $\chi$ runs over all even Dirichlet characters modulo $20$ (a Dirichlet character $\chi$ is even if $\chi(-1)=1$) and the formula $L(-1,\chi)=-B_{2,\chi}/2$ in terms of the generalized Bernoulli number $B_{2,\chi}$ associated with $\chi$ (\cite[Thm. 4.2, 4.3]{WashingtonCF}). The general formula for the values $B_{2,\chi}$ can be found for instance in \cite[Prop. 4.1]{WashingtonCF}. If $\chi$ is a non-trivial even character with conductor $N$ there is a simple expression for $B_{2,\chi}$, namely $B_{2,\chi}=\frac{1}{N}\sum_{a=1}^N \chi(a)a^2$ (see for instance \cite[Exercise 4.2]{WashingtonCF}). Now, there are four even Dirichlet characters mod $20$: the trivial character $\chi_0$ with $B_{2,\chi_0}=1/6$, a real character $\chi_1$ with conductor $5$ given by $\chi_1(11)=1$, $\chi_1(17)=-1$ contributing $B_{2,\chi_1}=4/5$ and two conjugate even characters $\chi_2$, $\bar \chi_2$ with conductor $20$ defined by $\chi_2(11)=-1$, $\chi_2(17)=\sqrt{-1}$ contributing $B_{2,\chi_2}=4\sqrt{-1}+8$. Altogether we get $\zeta_k(-1)=\frac{1}{12}\frac{2}{5}(2\sqrt{-1}+4)(-2\sqrt{-1}+4)$=2/3.  
\end{proof}

Let now $A=A(k;4,4,\pP_2,\pP_5)$ be the quaternion algebra over $k$, ramified exactly at the two places $\pP_2$ and $\pP_5$. 

\begin{thm}
\label{example1}
Let $\OO\subset A(k;4,4,\pP_2,\pP_5)$ be a maximal order. Then, $X_{\Gamma_{\OO}^1}$ is a four-dimensional fake product of projective lines.
\end{thm}

\begin{proof}
By Lemma \ref{euler_zahl}, $e(X_{\Gamma_{\OO}^1})=2\zeta_k(-1)(N\pP_2-1)(N\pP_5-1)$, and Lemma \ref{zerlegung1} gives immediately the asked value $e(X_{\Gamma_{\OO}^1})=2\cdot\frac{2}{3}(2^2-1)(5-1)=16=2^4$. We need to prove that $X_{\Gamma_{\OO}^1}$ is smooth. For this we will exclude the existence of elements of finite order in $\Gamma_{\OO}^1$. Note that there exists a torsion element in $\Gamma_{\OO}^1$ if and only if there is an embedding of $\OO_k[\rho]$ into $\OO$, where $\rho$ is some primitive root of unity. In particular, every torsion corresponds to a (commutative) subfield $k(\xi)\subset A$. In order to show the claim we will exclude the possibility of an embedding $k(\rho)\hookrightarrow A$. It is sufficient to consider only primitive $p$-th roots of unity $\rho$, where $p$ is an odd prime or $p=4$. Since $[k:\QQ]=4$, and $k(\rho)/k$ is at most a quadratic extension, $k(\rho)\subset A$ is possible at most for $p=3,4$ or $5$. By the classical embedding theorem of Hasse \cite[Theoreme 3.8, p.78]{VignerasAlgebres}, $k(\rho)$ can be embedded in $A$ if and only if each prime $\pP\mid d_A$ is non-split in $k(\rho)$. Let us consider $\rho=\xi_3$ and show that $\pP_2\subset \OO_k$, the prime ideal over $2$, splits in the field $L=k(\xi_3)$. The field $L=k(\sqrt{-3})$ is a subfield of $\QQ(\xi_{60})$, and by the facts from the general theory of cyclotomic fields used in the proof of Lemma \ref{zerlegung1}, $2\OO_{\QQ(\xi_{60})}=\mathfrak Q_1^2\mathfrak Q_2^2$, where $\mathfrak Q_{1,2}$ are both of inertia degree $4$ over $2$. The primes $\mathfrak Q_1$ and $\mathfrak Q_2$ are Galois conjugate by the automorphism $\sqrt{-3}\mapsto -\sqrt{-3}$, since $2\OO_{\QQ(\xi_{20})}=\qQ^2$ with inertia degree $f(\qQ/2)=4$. Lemma \ref{zerlegung1} implies that $\pP_2$ is split in $k(\xi_3)$ and therefore $k(\xi_3)$ is not a subfield of $A$.\\
Consider as next $L=k(\rho)$ with $\rho=\xi_4$ or $\rho=\xi_5$. Then, $L=K=k(\xi_{20})$, and we can use same kind of arguments:\\ Namely, we know that $K/k$ is unramified at all the finite places and $5\OO_K=\mathfrak P^4\cdot\mathfrak P'^4$ with two prime ideals $\mathfrak P$, $\mathfrak P'$ in $K$ with $f(\mathfrak P/5)=f(\mathfrak P'/5)=1$. By Lemma \ref{zerlegung1}, $5\OO_k=\pP_5^4$, hence $\pP_5=\mathfrak P\mathfrak P'$ is split in $K$ and $K$ is not a subfield of $A$.
\end{proof}

Let us present another example. Let $\xi_{24}$ be a primitive $24$-th root of unity, $L=\QQ(\xi_{24})$ the corresponding cyclotomic field and $\ell=\QQ(\xi_{24}+\xi_{24}^{-1})$ the maximal totally real subfield of $L$. Then, as in the example before we have the following elementary result, which can be proved in the same way as Lemma \ref{zerlegung1}.
\begin{lemma}
\label{zerlegung2}
$\ell$ is a quartic field of discriminant $d_{\ell}=2304$. $\zeta_{\ell}(-1)=1$. We have the following prime ideal decomposition of rational primes in $\OO_{\ell}$:
\begin{itemize}
\item $2\OO_{\ell}=\pP_2^4$
\item $3\OO_{\ell}=\pP_3^2$
\item $5\OO_{\ell}=\pP_5\pP_5'$
\end{itemize}
\end{lemma}
\begin{thm}
 \label{example2}
Let $A=A(\ell;4,\pP_2,\pP_3)$ be the totally indefinite quaternion algebra over $\ell$ of reduced discriminant $d_A=\pP_2\pP_3$ and $\OO$ a maximal order in $A$. Then, $X_{\Gamma_{\OO}^1}$ is an irreducible four-dimensional fake product of projective lines. 
\end{thm}
\begin{proof}
The proof goes along the lines of the proof of Theorem \ref{example1}. The equality $e(X_{\Gamma_{\OO}^1})=2(2-1)(3^2-1)=16$ follows immediately from Lemma \ref{zerlegung2} and Lemma \ref{euler_zahl}. For proving the smoothness of $X_{\Gamma_{\OO}^1}$, note that $\ell(\xi_3)$ and $\ell(\xi_4)$ coincide with the cyclotomic field $\QQ(\xi_{24})$ where $\pP_3\subset \OO_{\ell}$, the prime ideal over $3$ is split. This shows that $\ell(\xi_3)$ and $\ell(\xi_4)$ cannot be contained in $A$. Thus, $X_{\Gamma_{\OO}^1}$ is smooth.     
\end{proof}

After giving the examples, let us show that all the other virtual candidates from Theorem \ref{candidates} do not give rise to a lattice of a fake product.

\begin{thm}
\label{noexample3}
Besides the two examples above, no other quaternion algebra $A$ contains a lattice of a fake $(\PP^1)^n$ which is a finite index subgroup of the norm-one group of a maximal order in $A$.
\end{thm}
\begin{proof}
We will exclude the only possible quaternion algebras from Theorem \ref{candidates} by showing that the corresponding lattices are never torsion-free (if they exist). Consider for instance $A(k_{1957},4,4,\pP_3\pP_7)$ with $k_{1957}$ totally real quartic field with discriminant $1957$, which has the smallest discriminant among the possible fields. From the Table \ref{grad4} in Section \ref{tables} we deduce that $vol(\Gamma_{\OO}^1)=\frac{4}{3}\cdot (3-1)(7-1)=16$, hence $\Gamma_{\OO}^1$ is the only candidate for a lattice of a $4$-dimensional fake product. But we can show that $\Gamma_{\OO}^1$ contains an element of order $2$ which comes from an embedding of $\xi_4$ into $\OO$. Namely, let $K=k_{1957}(\xi_4)$, then $K$ is a totally complex field of discriminant $980441344=2^8 19^2 103^2$ defined by $x^8-4x^6+2x^5+16x^4-12x^3+ 25x^2-12x+37$ (PARI). Considering the factorization of this polynomial modulo $3$, we find that the rational prime $3$ factorizes in $K$ as $\qQ_3\qQ_3'$ as a product of two prime ideals. Since $3$ is also product $\pP_3\pP_3'$ of two prime ideals in $k$ (compare Table \ref{grad4}), $\pP_3$ remains prime in $K$ and hence by the embedding theorem \cite[Theoreme 3.8, p.78]{VignerasAlgebres} $\xi_4$ can be embedded into $A$. But this embedding can even be chosen in such a way that $\xi_4$ lies in $\OO$ \cite[Proposition 2.8]{Shimura67}, so that $\xi_4$ gives rise to a torsion in $\Gamma_{\OO}^1$. In all the other cases, except $A=A(k_{453789}, 6,4,\emptyset)$, the same kind of arguments work: the only possible lattice is $\Gamma_{\OO}^1$ itself, but which is never torsion-free. In case $A=A(k_{453789}, 6,4,\emptyset)$, where $k_{453789}$ is the totally real sextic field of discriminant $453789$, the possible lattice of a 4-dimensional fake product is a subgroup of index $6$ in $\Gamma_{\OO}^1$.  Now, recall the following elementary group theoretic result (see \cite[IX 9.2]{BrownCohomology} for a proof of a more precise result): If $H, G'<G$ are groups with $H$ finite and $G'$ torsion-free, then the index $[G:G']$ is divisible by $|H|$. On the other hand, $k_{453789}=\QQ(\xi_{21}+\xi_{21}^{-1})$ is the maximal totally real number field of the cyclotomic field $\QQ(\xi_{21})$. Since $\QQ(\xi_{21})$ is totally imaginary quadratic extension of $k_{453789}$ and $A$ is unramified at all the finite places, $\xi_{21}$ can be embedded into $\OO$ and the index of a torsion-free subgroup in $\Gamma_{\OO}^1$ must be divisible by $21$ which contradicts the assumption $[\Gamma_{\OO}^1:\Gamma]=6$.    
\end{proof}
\begin{remark}
More precisely we can say that up to isomorphism, the two examples of fake $(\PP^1)^4$ are the only examples of irreducible fake $(\PP^1)^n$ with $n\geq 4$ whose fundamental group is contained in $\Gamma_{\OO}^1$. This follows from the fact that there is a single conjugacy class of maximal orders inside $A(\QQ(\xi_{20}+\xi_{20}^{-1}), 4,4,\pP_{2}\pP_5)$ and $A(\QQ(\xi_{24}+\xi_{24}^{-1}), 4,4,\pP_{2}\pP_3)$ since both fields have class number one (see \cite[Section 6.7, (6.13)]{MaclachlanReid03}).  
\end{remark}

\section{Tables} 
\label{tables}
In this section we list the tables of invariants of quaternion algebras related to lattices of fake products of projective lines used in previous sections. The invariants are obtained as follows. First we collect all defining polynomials of number fields of degree $m\leq 6$ with root discriminant $< f(m)$ which can be found in \cite{NF} and \cite{Voighthomepage}. The PARI procedure \verb+nfinit->idealprimedec+ gives  the decomposition of small rational primes in $k$. We use the PARI procedure \verb+zetakinit->zetak+ to compute the values $\zeta_k(-1)$. More precisely, PARI apriori computes an approximation of the true value of $\zeta_k(2)$ by computing the truncated Euler product $\prod_{p\leq x}\prod_{\pP\mid p} (1-N\pP^{-2})^{-1}$. The accuracy depends on the choice of $x$ but the precision of $5-10$ decimal digits is certain and this turns out to be enough in given cases. We can namely control the result knowing some properties of the zeta values $\zeta_k(-1)$. Recall, that by Theorem of Klingen-Siegel (\cite{Siegel1969}) $\zeta_k(-1)$ is a rational number and moreover we have information on primes dividing the denominator of $\zeta_k(-1)$ and its size. By \cite[p. 89]{Siegel1969} we know that the denominator of $2^{m}\zeta_k(-1)$ divides a certain integer $c_{2m}$ which in given cases is product of small primes (if $m=4$ then $p=2,3,5$). 

Hence we have an upper bound $B$ such that $B\zeta_k(-1)$ is an integer.  
The functional equation implies that $B \zeta_k(-1)=B\zeta_k(2)\cdot (-1)^m2^{-m}\pi^{-2m}d_k^{3/2}$ is an integer. So with an approximation $Z_k(2)$ we would take the closest integer to $BZ_k(2)\cdot (-1)^m2^{-m}\pi^{-2m}d_k^{3/2}$ to obtain the true value. Consider for instance $k$ the quartic field with discriminant $d_k=725$. PARI gives us the value $Z_k(2)=1.0369329880\ldots$. By Siegel's theorem we know that only $2,3,5$ can divide the denominator of $\zeta_k(-1)$. Hence from the value $30Z_k(2)2^{-4}\pi^{-8}725^{3/2}=3.9999999997$ we find that $\zeta_k(-1)=4/30=2/15$ is the correct value. 

\begin{table}[h]
\centering

\begin{tabular}{|@{}c@{}|@{}c@{}|@{}c@{}|@{}c@{}|@{}c@{}|@{}c@{}|@{}c@{}|@{}c@{}|@{}c@{}|}
\hline
$d_k$  & defining polynomial & $\zeta_k(-1)$ &  $f(\pP/2)$ & $f(\pP/3)$ & $f(\pP/5)$ & $f(\pP/7)$ & $f(\pP/11)$ & $f(\pP/13)$\\
\hline
$725 $ & $x^4 - x^3 - 3x^2 + x + 1$ 	&  $2/15$ & $4$ & $4$ & $2$ & $2$ ; $2$ & $1$ ; $1$ ; $2$ & $2$ ; $2$\\
\hline
$1125$ & $x^4 - x^3 - 4x^2 + 4x + 1$	&  $4/15$ & $4$ & $2$ & & $4$ & $2$ ; $2$ & \\
\hline
$1957$ & $x^4 - 4x^2 - x + 1$		&  $2/3$ & $4$ & $1$ ; $3$ & $4$ & $1$ ; $3$ & & \\
\hline
$2000$ & $x^4 - 5x^2 + 5$ 		&  $2/3$ & $2$ & $4$ & $1$ & $4$ & & $4$\\
\hline
$2225$ & $x^4 - x^3 - 5x^2 + 2x + 4$	&  $4/5$ & $2$ ; $2$ & $4$ & & & $2$ ; $2$ & \\
\hline
$2304$ & $x^4 - 4x^2 + 1$		&  $1$ & $1$ & $2$ & $2$  $2$ & & & \\
\hline
$2525$ & $x^4 - 2x^3 - 4x^2 + 5x + 5$	&  $4/3$ & $4$ & $4$ & $1$ ; $1$ & $4$ & & \\
\hline
$2624$ & $x^4 - 2x^3 - 3x^2 + 2x + 1$	&  $1$ & $2$ & $4$ & & $1$ ; $1$ ; $2$ & & \\
\hline
$2777$ & $x^4 - x^3 - 4x^2 + x + 2$	&  $4/3$ & $1$ ; $3$ & $4$ & & $4$ & & \\
\hline
$3600$ & $x^4 - 2x^3 - 7x^2 + 8x + 1$	&  $8/5$ & & & & & & \\
\hline
$3981$ & $x^4 - x^3 - 4x^2 + 2x + 1$	&  $2$ & $4$ & $1$ ; $2$ & $1$ ; $3$ & & & \\
\hline
$4205$ & $x^4 - x^3 - 5x^2 - x + 1$	&  $2$ & $4$ & $4$ & $1$ ; $2$ & & & \\
\hline
$4352$ & $x^4 - 6x^2 - 4x + 2$		& $8/3$ & $1$ & $4$ & & & & \\
\hline
$9909$ & $x^4 - 6x^2 - 3x + 3$		& $8$ & $4$ & & & & & \\
\hline
$10512$ & $x^4 - 7x^2 - 6x + 1$		& $8$ & $2$ & & & & & \\
\hline
\end{tabular}

\caption{Totally real quartic fields with root discriminant $\leq 10.3$ and with the property that $16/2\zeta_k(-1)$ is an integer. Each entry in the column $f(\pP/p)$ is the inertia degree of a prime ideal $\pP$ dividing $p\OO_k$; the number of entries is the number of prime ideals dividing $p\OO_k$.}\label{grad4}  
\end{table}

\begin{table}[ht]
\centering
\begin{tabular}{|@{}c@{}|@{}c@{}|@{}c@{}|@{}c@{}|@{}c@{}|@{}c@{}|@{}c@{}|@{}c@{}|@{}c@{}|@{}c@{}|}
\hline
$d_k$ & defining polynomial & $\zeta_k(-1)$ & $f(\pP/2)$ & $f(\pP/3)$ & $f(\pP/5)$ & $f(\pP/7)$ & $f(\pP/11)$ & $f(\pP/13)$\\
\hline
$ 24217$ & $ x^5 - 5x^3 - x^2 + 3x + 1$ & $-4/3 $ & $5$ & $5$ & $1$ ; $4$ & $5$ & $2$ ; $3$ & $2$ ; $3$\\
\hline
$ 36497$ & $x^5 - 2x^4 - 3x^3 + 5x^2 + x - 1$ & $-8/3$ & $5$ & $5$ & $2$ ; $3$ & $2$ ; $3$ & &\\
\hline
$38569 $ & $ x^5 - 5x^3 + 4x - 1$ & $-8/3$ & $5$ & $5$ & $5$ & $1$ ; $4$ & &\\
\hline
$81509$ & $x^5 - x^4 - 5x^3 + 3x^2 + 5x - 2$ & $-32/3$ & & & & & &\\
\hline
\end{tabular} 
\caption{Totally real quintic fields with root discriminant $\leq 10.570$ and with the property that $16/\zeta_k(-1)$ is an integer. Each entry in the column $f(\pP/p)$ is the inertia degree of a prime ideal $\pP$ dividing $p\OO_k$; the number of entries is the number of prime ideals dividing $p\OO_k$.}\label{grad5}  
\end{table}

\begin{table}[ht]
\centering
\begin{tabular}{|@{}c@{}|@{}c@{}|@{}c@{}|@{}c@{}|@{}c@{}|@{}c@{}|@{}c@{}|@{}c@{}|}
\hline
$d_k$ & defining polynomial & $\zeta_k(-1)$ & $f(\pP/2)$ & $f(\pP/3)$ & $f(\pP/5)$ & $f(\pP/7)$ & $f(\pP/11)$\\
\hline
$300125$ & $x^6-x^5-7x^4+2x^3+7x^2-2x-1$ & $296/105$ & $6$ &  $6$ & $3$ & $2$ & $3$ ; $3$\\
\hline
$371293$ & $ x^6-x^5-5x^4+4x^3+6x^2-3x-1$ & $152/39$ & $6$ & $3$ ; $3$ & $2$ ;$2$ ; $2$ & $6$ &\\
\hline
$434581$ & $ x^6-2x^5-4x^4+5x^3+4x^2-2x-1$ & $104/21$ & $6$ & $3$ ; $3$ & $3$ ; $3$ & $2$ &\\
\hline
$ 453789$ & $x^6-x^5-6x^4+6x^3+8x^2-8x+1$ & $ 16/3$ & $6$ & $3$ & $3$ ; $3$ & $1$ &\\
\hline
$ 485125$ & $ x^6-2x^5-4x^4+8x^3+2x^2-5x+1$ & $88/15$ & $6$ & $2$ ; $4$ & $3$ & &\\
\hline
$ 592661$ & $x^6-x^5-5x^4+4x^3+5x^2-2x-1 $ & $ 8 $ & $6$ & $6$ & $2$ ; $4$ & &\\
\hline
$ 703493$ & $ x^6-2x^5-5x^4+11x^3+2x^2-9x+1$ & $72/7 $ & $6$ & $6$ & & &\\
\hline
$ 722000$ & $ x^6-x^5-6x^4+7x^3+4x^2-5x+1 $ & $56/5$ & $2$ & $6$ & & &\\
\hline
$ 810448$ & $ x^6 -3x^5-2x^4+9x^3-5x+1 $ & $ 40/3 $ & $2$ & $3$ ; $3$ & & &\\
\hline
$ 820125$ & $ x^6 - 9x^4 - 4x^3 + 9x^2 + 3x - 1 $ & $ 584/45 $ & $6$ & $2$ & & &\\
\hline
$ 905177$ & $ x^6-x^5-7x^4+9x^3+7x^2-9x-1 $ & $ 320/21$ & $3$ ; $3$ & $6$ & & &\\
\hline
$ 966125$ & $ x^6-x^5-6x^4+4x^3+8x^2-1 $ & $ 256/15 $ & $6$ &  & & &\\
\hline
$980125$ & $x^6-x^5-6x^4+6x^3+7x^2-5x-1 $ & $256/15$ & $6$ &  & & &\\
\hline
$ 1075648$ & $x^6 - 7x^4 + 14x^2 - 7$ & $416/21$ & $3$ &  &  & &\\
\hline
$ 1081856$ & $ x^6-6x^4-2x^3+7x^2+2x-1$ & $ 20 $ & $3$ &  &  & &\\
\hline
$ 1134389$ & $x^6-2x^5-4x^4 + 6x^3 + 4x^2 - 3x - 1 $ & $ 64/3 $ & $6$ &  & & &\\
\hline
$ 1202933$ & $x^6 - 6x^4 - 2x^3 + 6x^2 + x - 1 $ & $ 24 $ & $6$ &  & & & \\
\hline
$ 1229312$ & $x^6 - 10x^4 + 24x^2 - 8 $ & $ 172/7 $  & $3$ &  & & &\\
\hline
$ 1241125$ & $x^6 - 7x^4 - 2x^3 + 11x^2 + 7x + 1 $ & $ 376/15 $ & $6$ &  & & &\\
\hline
$ 1259712$ & $x^6 - 6x^4 + 9x^2 - 3 $ & $ 248/9 $  & $3$ & &  & & \\
\hline
$ 1279733 $ & $x^6 - 2x^5 - 6x^4 + 10x^3 + 10x^2 - 11x - 1 $ & $ 544/21 $ & $6$ & & & &\\
\hline
$ 1292517 $ & $x^6 - 6x^4 - x^3 + 6x^2 - 1 $ & $ 232/9 $ & $6$ &  & & & \\
\hline
$ 1312625 $ & $x^6 - x^5 - 7x^4 + 7x^3 + 12x^2 - 12x - 1 $ & $ 416/15 $ & $2$ ; $4$ &  & & &\\
\hline
$ 1387029 $ & $x^6 - 3x^5 - 2x^4 + 9x^3 - x^2 - 4x + 1  $ & $ 32 $ & $6$ &  & & &\\
\hline
$ 1397493 $ & $x^6 - 3x^5 - 3x^4 + 10x^3 + 3x^2 - 6x + 1  $ & $ 32 $ & $6$ &  & & &\\
\hline
$ 1416125 $ & $x^6-2x^5-5x^4+9x^3+6x^2-9x+1  $ & $ 152/5 $  & $2$ & & & &\\
\hline
$ 1528713 $ & $x^6-3x^5-3x^4+7x^3+3x^2-3x-1 $ & $ 304/9 $ & $3$ ; $3$ & & & &\\
\hline
\end{tabular} 
\caption{Totally real sextic fields with root discriminant $\leq 10.734$. The entries should be read as follows: Each entry in the column  $f(\pP/p)$ is the inertia degree of a prime ideal $\pP$ dividing $p\OO_k$; the number of entries is the number of prime ideals dividing $p\OO_k$.}\label{grad6}  
\end{table}
\clearpage

\bibliographystyle{amsplain}

\end{document}